\documentclass{amsart}

\usepackage{amsmath}   
\usepackage{amssymb}   
\usepackage{dsfont}    
\usepackage{bm}        
\usepackage{mathtools} 
\usepackage[hypertexnames=false]{hyperref} 
\usepackage{enumerate}
\usepackage{centernot}
\usepackage{stackrel}
\usepackage[english]{babel}

\usepackage[numbers]{natbib}
\usepackage{bibunits}

\theoremstyle{plain}
\newtheorem{proposition}{Proposition}
\theoremstyle{remark}
\newtheorem{remark}{Remark}

\makeatletter
\newcommand{\addresseshere}{%
  \enddoc@text\let\enddoc@text\relax
}

\begin{document}

\title[Rotational Uniqueness Conditions]
{Rotational Uniqueness Conditions Under \\ Oblique Factor Correlation Metric}

\author[C.F.W.\ Peeters]{Carel F.W.\ Peeters}
\address[Carel F.W.\ Peeters]{
Dept.\ of Epidemiology \& Biostatistics \\
Amsterdam Public Health research institute \\
VU University medical center Amsterdam \\
Amsterdam\\
The Netherlands}
\email{cf.peeters@vumc.nl}

\begin{abstract}\label{abstract}
\noindent In an addendum to his seminal 1969 article J\"{o}reskog
stated two sets of conditions for rotational identification of the
oblique factor solution under utilization of fixed zero elements in
the factor loadings matrix \citep{Jog79}. These condition sets,
formulated under factor correlation and factor covariance metrics,
respectively, were claimed to be equivalent and to lead to global
rotational uniqueness of the factor solution. It is shown here that
the conditions for the oblique factor correlation structure need to
be amended for global rotational uniqueness, and hence, that the
condition sets are not equivalent in terms of unicity of the
solution.

\end{abstract}

\maketitle

\section{Introduction}\label{IntIdent}
Suppose
$\mathbf{\Sigma}=\mathbf{\Lambda}\mathbf{\Phi}\mathbf{\Lambda}^{\mathrm{T}}+\mathbf{\Psi}$
is the usual oblique factor analysis model. Here,
$\mathbf{\Lambda}\in\mathbb{R}^{p\times m}$ is a matrix of factor
loadings in which each element $\lambda_{jk}$ is the loading of the
$j$th variable on the $k$th factor, $j=1,\ldots,p$, $k=1,\ldots,m$;
$\mathbf{\Phi}\in\mathbb{R}^{m\times m}$ denotes the factor
covariance matrix; and $\mathbf{\Psi}\in\mathbb{R}^{p\times p}\equiv
\mbox{diag}[\psi_{11},\ldots, \psi_{pp}]$ contains the error
variances. In the remainder we will assume the usual regularity
assumptions: (a) rank$(\mathbf{\Lambda})=m$; (b) $\psi_{jj}>0
~\forall j$; and (c) $(p-m)^{2}-p-m\geqslant0$, simply stating
nonnegative degrees of freedom for existence.

As is well-known, the factor model is inherently underidentified,
implying that $\mathbf{\Sigma}$ does not have a unique solution
without imposing restrictions. Given $\mathbf{\Psi}$, two factor
models defined by $\{\mathbf{\Lambda},\mathbf{\Phi}\}$ and
$\{\mathbf{\Lambda}^{\ddag},\mathbf{\Phi}^{\ddag}\}$ are equivalent
if there exists a mapping
$\delta:\{\mathbf{\Lambda},\mathbf{\Phi}\}\longrightarrow\{\mathbf{\Lambda}^{\ddag},\mathbf{\Phi}^{\ddag}\}$
such that
$\mathbf{\Sigma}(\mathbf{\Lambda},\mathbf{\Phi},\mathbf{\Psi})=
\mathbf{\Sigma}[\delta(\mathbf{\Lambda},\mathbf{\Phi}),\mathbf{\Psi}]$.
For the factor model we find
$\mathbf{\Sigma}=(\mathbf{\Lambda}\mathbf{R})[\mathbf{R}^{-1}\mathbf{\Phi}(\mathbf{R}^{\mathrm{T}})^{-1}]
(\mathbf{\Lambda}\mathbf{R})^{\mathrm{T}}+\mathbf{\Psi}$, where
$\mathbf{R}\in\mathbb{R}^{m\times m}$ is an arbitrary nonsingular
matrix, implying that there is an infinite number of alternative
matrices $\mathbf{\Lambda}^{\ddag}=\mathbf{\Lambda}\mathbf{R}$ and
$\mathbf{\Phi}^{\ddag}=\mathbf{R}^{-1}\mathbf{\Phi}(\mathbf{R}^{\mathrm{T}})^{-1}$
that generate the same covariance structure $\mathbf{\Sigma}$. The
operation $\mathbf{\Lambda}\mapsto\mathbf{\Lambda}\mathbf{R}$ is
termed `rotation', and $\mathbf{\Lambda}$ and $\mathbf{\Phi}$ are
said to be globally rotationally unique \emph{iff}
$\mathbf{R}=\mathbf{I}_{m}$.

Finding conditions to ensure (global) rotational uniqueness has been
an active area of research and debate in the factor analysis
community. Building on Howe \citep{How55}, J\"{o}reskog
\citep{jog69} conjectured sufficient conditions for uniqueness
for both the orthogonal ($\mathbf{\Phi}=\mathbf{I}_{m}$) and oblique
factor model through specification of fixed elements in
$\mathbf{\Lambda}$ and $\mathbf{\Phi}$. Dunn \citep{Dun73}
showed, especially for the orthogonal model, that J\"{o}reskog's
conditions were not sufficient. His substitute conditions based on
fixed zero elements are sufficient for local rotational uniqueness
only, in the sense that $2^{m}$ combinations of polarity reversals
in the columns of $\mathbf{\Lambda}$ are allowed, giving that
$\mathbf{R}$ is then $\mbox{diag}[\pm1,\ldots, \pm1]$. Jennrich
\citep{Jenn78} gave sufficient conditions for local rotational
uniqueness under reflections for the orthogonal model when the fixed
elements are arbitrary. These works inspired J\"{o}reskog to write
an addendum to his 1969 article \citep{Jog79}, focussing on
reformulating the sufficiency conditions for the oblique factor
solution with fixed zero elements. He gave the following conditions:
\newcounter{Zcount}
\begin{list}{(C\arabic{Zcount})}
  {\usecounter{Zcount}}
  \item Let $\mathbf{\Lambda}$ have at least $m-1$ fixed zeroes in each column;
  \item Let rank$(\mathbf{\Lambda}^{[k]})=m-1$, where $\mathbf{\Lambda}^{[k]}$, $k=1,\ldots,m$, is the submatrix of
$\mathbf{\Lambda}$, consisting of the rows of $\mathbf{\Lambda}$
which have fixed zero elements in the $k$th column with these zeroes
deleted;
  \item Let $\mathbf{\Phi}$ be a symmetric positive
definite matrix with $\mbox{diag}(\mathbf{\Phi})=\mathbf{I}_{m}$
(i.e., $\mathbf{\Phi}$ is a correlation matrix);
\end{list}
and conjectured that C1-C3 are sufficient for obtaining global
rotational uniqueness. Moreover, he conjectured that conditions C1
and C3 are equivalent to:
\newcounter{Ycount}
\begin{list}{(C*)}
  {\usecounter{Ycount}}
  \item Let $\mathbf{\Lambda}$ have at least $m-1$ fixed
zeroes in each column and one fixed non-zero value in each column,
the latter values being in different rows;
\end{list}
and subsequently proved global rotational uniqueness under pairing
of conditions C2 and C*.

Many recent texts follow J\"{o}reskog \citep{Jog79} in stating
that conditions C1-C3 are sufficient for (rotational) uniqueness
\cite[e.g.,][]{HD04,AM09}. However, it will be shown that conditions
C1-C3 are \emph{not} sufficient for global rotational uniqueness but
local rotational uniqueness only and, hence, that conditions C1 and
C3 are \emph{not} equivalent to C* in terms of unicity of the
solution. Although this result may be implicitly known or be
considered tacit knowledge, here it is made explicit.

In the remainder, condition set C1-C3 will be amended with an
additional condition. It will then be shown that the amended
condition set is sufficient for obtaining global rotational
uniqueness, implying that C1-C3 do not lead to global uniqueness and
the non-equivalence of conditions C1 and C3 to C*. Section
\ref{Disc} concludes with a discussion.

\section{Global Rotational Uniqueness Under $\mbox{diag}(\mathbf{\Phi})=\mathbf{I}_{m}$}\label{Cond}
Consider the following addition to conditions C1-C3:
\newcounter{Xcount}
\begin{list}{(C4)}
  {\usecounter{Xcount}}
  \item Let in each column of $\mathbf{\Lambda}$ one
parameter non-fixed by condition C1 be polarity truncated to take
only positive or negative values, that is: In each column of
$\mathbf{\Lambda}$ one element is to adopt either strict positivity
$(\lambda_{jk}>0)$, or strict negativity $(-\lambda_{jk}>0)$.
\end{list}

\begin{proposition}\label{UniquePropC3}
Let the mapping
$\delta:\{\mathbf{\Lambda},\mathbf{\Phi}\}\longrightarrow\{\mathbf{\Lambda}^{\ddag},\mathbf{\Phi}^{\ddag}\}$
be defined by $\mathbf{\Lambda}^{\ddag}=\mathbf{\Lambda}\mathbf{R}$
and
$\mathbf{\Phi}^{\ddag}=\mathbf{R}^{-1}\mathbf{\Phi}(\mathbf{R}^{\mathrm{T}})^{-1}$,
where $\mathbf{R}\in\mathbb{R}^{m\times m}$ denotes an arbitrary
nonsingular matrix. If conditions \emph{C1}-\emph{C4} hold, then
$\mathbf{R}=\mathbf{I}_{m}$.
\end{proposition}

\begin{proof}
\begin{sloppypar}
Let conditions C1-C4 hold on $\{\mathbf{\Lambda},\mathbf{\Phi}\}$.
We start by showing that $\mathbf{R}$ is diagonal under conditions
C1-C2, which can be shown with an argument analogous to
\citet[pp.~576-577]{AND84}. Under given conditions it is always
possible to find a permutation matrix $\mathbf{P}_{1}$ of respective
dimension $p \times p$, such that $\mathbf{P}_{1}\mathbf{\Lambda}$
gives a block lower right triangular form on
$\mathbf{\Lambda}^{[1]}$. We then have

\begin{equation}\label{blockmatr}
\mathbf{\Lambda}'=\mathbf{P}_{1}\mathbf{\Lambda}=\left[
                                     \begin{array}{cc}
                                       \boldsymbol{0}             & \mathbf{\Lambda}^{[1]} \\
                                       \boldsymbol{\lambda}_{(1)} & \mathbf{\Lambda}_{[1]} \\
                                     \end{array}
                                   \right],
\end{equation}
where $\boldsymbol{0}$ is a $(m-1)$-dimensional null vector,
$\boldsymbol{\lambda}_{(1)} \in\mathbb{R}^{(p-m+1)\times 1}$,
$\mathbf{\Lambda}^{[1]}\in\mathbb{R}^{(m-1)\times (m-1)}$, and
$\mathbf{\Lambda}_{[1]}\in\mathbb{R}^{(p-m+1)\times (m-1)}$. Now,
let
\begin{equation}\nonumber
\mathbf{R}=\left[
                 \begin{array}{cc}
                 r_{11}                       & \boldsymbol{\mathrm{r}}_{12} \\
                 \boldsymbol{\mathrm{r}}_{21} & \mathbf{R}_{22} \\
                 \end{array}
                 \right],
\end{equation}
where $r_{11}$ is a scalar,
$\boldsymbol{\mathrm{r}}_{12}\in\mathbb{R}^{1\times (m-1)}$,
$\boldsymbol{\mathrm{r}}_{21}\in\mathbb{R}^{(m-1)\times 1}$, and
$\mathbf{R}_{22}\in\mathbb{R}^{(m-1)\times (m-1)}$. Then
\begin{equation}\nonumber
(\mathbf{\Lambda}')^{\ddag}=\mathbf{\Lambda}'\mathbf{R}=\left[
                 \begin{array}{cc}
                 \mathbf{\Lambda}^{[1]}\boldsymbol{\mathrm{r}}_{21} & ~~\mathbf{\Lambda}^{[1]}\mathbf{R}_{22} \\
                 r_{11}\boldsymbol{\lambda}_{(1)}+\mathbf{\Lambda}_{[1]}\boldsymbol{\mathrm{r}}_{21} & ~~\boldsymbol{\lambda}_{(1)}\boldsymbol{\mathrm{r}}_{12}+\mathbf{\Lambda}_{[1]}\mathbf{R}_{22} \\
                 \end{array}
                 \right].
\end{equation}
As the rank of $\mathbf{\Lambda}^{[1]}$ is $m-1$,
$\boldsymbol{\mathrm{r}}_{21}$ should be a null vector. We may
follow the same procedure for each respective remaining submatrix
$\mathbf{\Lambda}^{[k]}$. That is, we can find an $(m\times
m)$-dimensional permutation matrix $\mathbf{P}_{2}$ such that
$\mathbf{P}'_{1}\mathbf{\Lambda}\mathbf{P}_{2}$ has a structure as
in (\ref{blockmatr}) with the index $1$ replaced by $k$ (we may
always find a pair of permutation matrices $\mathbf{P}'_{1}$ and
$\mathbf{P}_{2}$, such that
$\mathbf{P}'_{1}\mathbf{\Lambda}\mathbf{P}_{2}$ gives a block lower
right triangular form on $\mathbf{\Lambda}^{[k]}$). From the
accompanying reorderings
$\mathbf{P}_{2}^{\mathrm{T}}\mathbf{R}\mathbf{P}_{2}$ implied by the
identity
\begin{small}
\begin{align}\nonumber
&\mathbf{P}'_{1}\mathbf{\Sigma}(\mathbf{P}'_{1})^{\mathrm{T}}\\\nonumber&=
(\mathbf{P}'_{1}\mathbf{\Lambda}\mathbf{R})[\mathbf{R}^{-1}\mathbf{\Phi}(\mathbf{R}^{\mathrm{T}})^{-1}]
(\mathbf{P}'_{1}\mathbf{\Lambda}\mathbf{R})^{\mathrm{T}}+\mathbf{P}'_{1}\mathbf{\Psi}(\mathbf{P}'_{1})^{\mathrm{T}}\\\nonumber
&=(\mathbf{P}'_{1}\mathbf{\Lambda}\mathbf{P}_{2}\mathbf{P}_{2}^{\mathrm{T}}\mathbf{R}\mathbf{P}_{2})
[\mathbf{P}_{2}^{\mathrm{T}}\mathbf{R}^{-1}\mathbf{P}_{2}\mathbf{P}_{2}^{\mathrm{T}}\mathbf{\Phi}
\mathbf{P}_{2}\mathbf{P}_{2}^{\mathrm{T}}(\mathbf{R}^{\mathrm{T}})^{-1}\mathbf{P}_{2}]
(\mathbf{P}'_{1}\mathbf{\Lambda}\mathbf{P}_{2}\mathbf{P}_{2}^{\mathrm{T}}\mathbf{R}\mathbf{P}_{2})^{\mathrm{T}}
+\mathbf{P}'_{1}\mathbf{\Psi}(\mathbf{P}'_{1})^{\mathrm{T}},
\end{align}
\end{small}
it follows that under conditions C1-C2,
$\mathbf{R}=\mbox{diag}[r_{11},\ldots, r_{mm}]$.
\end{sloppypar}

The properties of diagonal matrices are such that now
$\mathbf{R}^{\mathrm{T}}=\mathbf{R}$ and
$\mathbf{R}^{-1}=\mbox{diag}[r_{11}^{-1},\ldots, r_{mm}^{-1}]$.
Superimposing condition C3 then implies
\begin{equation}\nonumber
\phi_{kk}^{\ddag}=r_{kk}^{-2}\phi_{kk}=r_{kk}^{-2}=1 ~~\forall k,
\end{equation}
giving that $r_{kk}=\pm1$, and $\mathbf{R}=\mbox{diag}[\pm1,\ldots,
\pm1]$.

By demanding that each column of $\mathbf{\Lambda}$ has a polarity
truncation, column multiplication by $-1$ is no longer possible as
this would imply a reversal of the polarity truncation (direction
inequality symbol). Superimposing condition C4 on C1-C3 thus gives
that $\mathbf{R}=\mathbf{I}_{m}$. The proposition follows.
\end{proof}

\begin{remark}
The proof of Proposition \ref{UniquePropC3} implies, contrary to
previous conjectures, that conditions C1-C3 are not sufficient for
global rotational uniqueness as they provide local rotational
uniqueness only and, hence, that conditions C1 and C3 are not
equivalent to C* as the pairing of C2-C* does provide global
rotational uniqueness \citep{Jog79}. Moreover, the proposition
indicates how inequality restrictions can aid in the attainment of
global rotational uniqueness.
\end{remark}

\begin{remark}
Instead of using strict positivity or strict negativity truncations
as formulated in condition C4 we could also use strict polarity
truncations by (arbitrary) constants, that is: Every column of
$\mathbf{\Lambda}$ should contain either
$\lambda_{jk}>c_{k}\in\mathbb{R}^{+}$ or
$-\lambda_{jk}>c_{k}\in\mathbb{R}^{+}$. While this will produce
global rotational uniqueness whence superimposed on C1-C3 along the
same lines as C4, it may not be practical in a research setting. It
may, for example, be possible to specify
$\lambda_{jk}>c_{k}\in\mathbb{R}^{+}$ while the true parameter value
$0\leqslant\lambda_{jk}<c_{k}$ (in the positive reflection). One
would then run into estimation trouble in numerical applications. A
related issue lies in choosing the loading elements for column
polarity fixation. For (Bayesian) estimation efficiency (see
Discussion section) condition C4 should be imposed on loadings that,
from prior knowledge or theory, are believed to be large.
\end{remark}

\begin{remark}
Please note that rotational uniqueness of $\mathbf{\Lambda}$ will
not guarantee identifiability of the FA model \citep{BJ85}, as
underidentification of $\mathbf{\Psi}$ may imply underidentification
of $\mathbf{\Lambda}$. However, if the regularity assumptions stated
in the introduction hold, then conditions C1-C4 will, next to
rotational uniqueness, also provide identifiability. If unsure if
identifiability is obtained, one could endeavor on algebraically
checking (local) identification utilizing the Wald rank rule
\citep{BMW94}.
\end{remark}

\section{Discussion}\label{Disc}
The following question deserves some exploration: What are reasons
to prefer condition set C1-C4 above the pairing C2-C*? Before
delving into possible answers it is stated why the addition of
condition C4 is deemed important when working in factor correlation
metric.

Not attaining global rotational uniqueness in factor correlation
metric under C1-C3 does not hamper maximum likelihood estimation, as
any local minimum has equivalent representations through simple
polarity reflections. In obtaining factor loading standard errors or
Bayesian estimates the situation is a little more intricate as the
parameter space under C1-C3 is multimodal. The modes (defined over
polarity reflections) should be widely separated in order for
unimodal normal approximations or resampling techniques to yield
valid estimates of the standard errors (see,
\citet[e.g.,][]{DolMol}). As Bayesian modeling proceeds through
exploration of posterior space, posterior estimates will be flawed
when transition probability between modes is non-negligible.
Imposing C4 can aid when the modes are not well-separated. In
Bayesian modeling, for example, imposing C4 implies a truncation of
the posterior density and will restrict posterior simulation to a
single mode -- a fact that has been recognized (for the orthogonal
factor model) by Geweke and Zhou \citep{GZ96}.

A first reason for preferring C1-C4 above the pairing C2-C* might be
found in the topic of arbitrary units of measurement. Many
(psychological) tests have units of measurement with no intrinsic
meaning. Let $\mathbf{D}$ be a diagonal matrix with positive
diagonal elements that indicate a change in test score units. Then
$\mathbf{D}\mathbf{\Sigma}\mathbf{D}\equiv\mathbf{\Sigma}^{\dag}$.
If $\mathbf{\Lambda}$ is identified by C1-C4, then
$\mathbf{D}\mathbf{\Lambda}\equiv\mathbf{\Lambda}^{\dag}$ is
similarly identified. If $\mathbf{\Lambda}$ is identified by the
pairing C2-C*, then each column of $\mathbf{D}\mathbf{\Lambda}$ has
to be renormalized \citep[p.~557]{AND84}.

A second reason for preference may be that the condition set C1-C4
is less restrictive on $\mathbf{\Lambda}$. Consider the unrestricted
factor model. Unrestricted solutions correspond to exploratory
factor analysis (EFA) in the sense that only minimal restrictions
are placed on the model to achieve at least a local rotationally
unique solution for $m$ factors. As such, an unrestricted solution
for $m$ common factors does not restrict the factor space and will
yield an optimal fit for any model with $m$ factors \citep[Section
15.4]{Mul10}. In effect, a minimal set of restrictions based on C2
and C* or C1-C4 entails a choice of rotation of the EFA model. The
pairing C2-C* imposes the minimum of $m^{2}$ restrictions, equalling
the number of non-redundant elements in $\mathbf{R}$, on the
parameter space of $\mathbf{\Lambda}$. The condition set C1-C4
imposes only $m(m-1)$ fixed-value restrictions on
$\mathbf{\Lambda}$. The parameters involved in the polarity
truncations are free to be estimated in either the positive or
negative range (note also that the polarity truncations do not
restrict the factor space). The fewer number of restrictions on
$\mathbf{\Lambda}$ under C1-C4 make it a more flexible set for
unrestricted formulations of the (confirmatory) factor model.

\section*{Acknowledgements}
This research was
supported by grant NWO-VICI-453-05-002 of the Netherlands
Organization for Scientific Research (NWO).
It was written while the author was a Ph.D. candidate at the Department of
Methodology and Statistics, Utrecht University, Utrecht, the Netherlands
and is part of the authors' Ph.D. thesis.
This version is a postprint of:
Peeters, C.F.W. (2012). Rotational Uniqueness Conditions under Oblique Factor
Correlation Metric. \emph{Psychometrika}, 77: 288--292.

\bibliographystyle{plainnat}
\bibliography{References}

\vspace{1cm}
\addresseshere

\end{document}